\documentclass{ws-ijac}

\def \bZ {{ \mathbb {Z}}}
\def \pO{{\mathcal{O}}}

\begin{document}

\markboth{O. Kharlampovich, A. Mohajeri}
{Approximation of Geodesics in  Metabelian Groups}

%%%%%%%%%%%%%%%%%%%%% Publisher's Area please ignore %%%%%%%%%%%%%%%
%
\catchline{}{}{}{}{}
%
%%%%%%%%%%%%%%%%%%%%%%%%%%%%%%%%%%%%%%%%%%%%%%%%%%%%%%%%%%%%%%%%%%%%

\title{APPROXIMATION OF GEODESICS IN METABELIAN GROUPS }

\author{OLGA KHARLAMPOVICH\footnote{
Department of Mathematics and Statistics , McGill University, Burnside Hall, Room 1005,
805 Sherbrooke Street West, Montreal, Quebec,  H3A 2K6, Canada} }

\address{Department of Mathematics and Statistics , McGill University, Burnside Hall, Room 1005,
805 Sherbrooke Street West\\
Montreal, Quebec,  H3A 2K6, Canada\\ 
\email{olga@math.mcgill.ca} }

\author{ATEFEH MOHAJERI MOGHADDAM\footnote{
Department of Mathematics and Statistics , McGill University, Burnside Hall, Room 1005,
805 Sherbrooke Street West, Montreal, Quebec,  H3A 2K6, Canada}}

\address{Department of Mathematics and Statistics , McGill University, Burnside Hall, Room 1005,
805 Sherbrooke Street West\\
Montreal, Quebec,  H3A 2K6, Canada\\ 
\email{mohajeri@math.mcgill.ca } }

\maketitle

\begin{history}
\received{(Day Month Year)}
\revised{(Day Month Year)}
\comby{[editor]}
\end{history}

\begin{abstract}
It is known that the bounded Geodesic Length Problem in free metabelian groups is NP-complete \cite{myasnikov} (in particular, the Geodesic Problem is $NP$-hard). We construct a  2-approximation polynomial time deterministic algorithm for the Geodesic Problem.  We  show that
 the Geodesic Problem in the restricted wreath product of a finitely generated non-trivial group with a finitely generated abelian group
containing  $\bZ^2$ is $NP$-hard and  there exists a Polynomial Time Approximation Scheme for this problem.
 We also show that  the Geodesic Problem in the restricted wreath product of two finitely generated non-trivial abelian groups
 is $NP$-hard if and only if the second abelian group contains $\bZ^2$.
\end{abstract}


 \maketitle

 \section{Introduction}

Studying  algorithmic problems in group theory could be traced back to Max Dehn.  In 1911, Dehn highlighted the $word$ $problem$ alongside two other problems: the Conjugacy Problem and the Group Isomorphism Problem. Since then, studying the algorithmic aspects of different problems has been an active field in combinatorial and geometric group theory. Algorithmic results on solvable groups were for a long time the real jems of combinatorial group theory, revealing interesting relations with computational commutative algebra and number theory. Researchers started looking again at the general classes of solvable groups, but this time from asymptotic, geometric, and computational perspectives.

In a recent work  Myasnikov $et$ $al$ \cite{myasnikov} consider the following three algorithmic problems related to Geodesic Problem in free solvable groups .

{\bf { \em The Geodesic Problem (GP) :}}  Let $X=\{ x_1, \dotsc,  x_m \}$ be a set of generators for a finitely generated group $G$ and $F(X)$ be the free group generated by $X$. Let $\mu : F(x) \rightarrow G $ be the canonical epimorphism. For a word $w \in F(X)$ represented in the alphabet $X^{\pm1}$ we denote by $|w|$  the length of $w$
  (number of generators used in a reduced presentation of $w$).  We define the Geodesic Length $l_X(g)$ of $g \in G$   with respect to the set of generators $X$ as follows

$$l_X(g):=min\{ |w|  \mid w \in F(X), \mu(w)=g \}$$

A word $w \in F(x)$ is called a {\it geodesic word} if $l_X(\mu(w))= |w|$.  The Geodesic Problem relative to $X$ is the following problem: given a word $g \in G$ find a {\it geodesic word} $w \in F(X)$  such that $g=\mu (w)$ in $G$.

{\bf{\em The Geodesic Length Problem (GLP) }}Let $G$ be a finitely generated group and $X$ be a set of generators of $G$. Given a word $w \in F(X)$ find $l_X(\mu(w))$.

{\bf{\em Bounded Geodesic Length Problem (BGLP) }}   Let $G$ be  a finitely generated group  with a set of generators $X$. Given a word $w \in F(X)$ and $k \in \mathbb{N}$ determine if $l_X(\mu(w)) \leq k$.

The hardness of $BGLP$, $GLP$ and $GP$ in a given group $G$ is discussed in \cite{myasnikov}. It is shown that each problem in the list is polynomial time {\it Turing reducible} to the next one. In the same paper $NP$-completeness of  $BGLP$ in free matabelian groups of arbitrary rank is proved.

We need to remind  some definitions from Complexity Theory.

{\bf{\em Polynomial-time Reducibility}} Problem $P$ is  Polynomial-time reducible to problem $Q$ if whenever $P$ has a solution,  such a solution can be obtained from a solution of $Q$ in polynomial time.

{\bf{\em Class NP}} Class $NP$ is the class of all decision problems for which the ''yes"-instances are recognizable in polynomial time by a non-deterministic Turing machine.

{\bf{\em NP-complete}} A problem $P$ is $NP$-complete if it is in class $NP$ and any other problem in class $NP$ is reducible to $P$ in polynomial-time. In other words, $P$ is at least as hard as any other problem in the class $NP$.

{\bf{\em  NP-hard  }} A problem $P$ is $NP$-hard if any $NP$-complete problem is polynomial time {\it Turing reducible} to $P$.

In a more recent paper \cite{elder}  Elder and Rechnitzer show that  for a finitely generated group $G$ these three problems are equivalent, i.e.,  they are reducible to each other in polynomial time.

  Motivated by some applications in group-based cryptography(see \cite{gp cryp}),  we construct a  2-approximation polynomial time deterministic algorithm for the Geodesic Problem in free metabelian groups. (Theorem 3.1). To the best of our knowledge, this is the first time that the idea of finding a good approximation to the optimal solution of  computationally hard problem is used in infinite group theory.  We  show that
 the Geodesic Problem in the restricted wreath product of a finitely generated group with a finitely generated abelian group
containing  $\bZ^2$ is $NP$-hard and  there exists a Polynomial Time Approximation Scheme for this problem. These approximation algorithms have an important role in solving the Conjugacy Problem in the so-called  $length$-$based$ $attack$ in  group-based cryptography( \cite{gp cryp}).
 We also show that  the Geodesic Problem in the restricted wreath product of two finitely generated  non-trivial abelian groups  is $NP$-hard if and only if the second group contains $\bZ^2$ (Theorems 2.3 and 2.4). This answers a question from \cite{myasnikov}.

We wish to thank A. Myasnikov and B. Shepherd for making many very helpful suggestions. We also thank L. Babai for showing that the TSP in virtually cyclic groups is in P, and S. Arora for giving a hint how to extend his polynomial time approximation schemes for the construction of Euclidean Traveling Salesman Tour to the construction of Traveling Salesman Path.

\section{Geodesic Problem In  Restricted Wreath Product  }

The results in this section are obtained by combining known results from group theory and computer science. Theorems 2.3, 2.4 and 2.6 are not hard to prove but the statements are interesting since they introduce the idea of Polynomial time Approximation Algorithm into infinite group theory. Furthermore,  they answer open questions and have never been formulated elsewhere.

Let us first agree about notations.  Let $A$ and $B$ be groups. Let

$$D:=\{ f:B \to A | |supp(f)| <  \infty \}.$$

 For $b \in B$ we  define the following action on $D$ :

\begin{align*}
&\hat{b}: D \to  D \\
& f(.) \mapsto f(b^{-1}.)
\end{align*}

We have an embedding $B \hookrightarrow Aut(D) $ by identifying $b$ with $\hat{b}$.

 The restricted wreath product $G=A \wr B $ is  the restricted semidirect  product $D \rtimes B$  consisting of all pairs $(f,b)$, $f \in D$ and  $b \in B$ with the operation

$$(f,b)(f',b'):=(fb(f'),bb').$$

For any $a \in  A$, let  $f_a \in D$  be the function that maps identity to $a$ and everything else to identity.  The map $a \mapsto (f_a,1 )$ gives the inclusion $A \hookrightarrow G$. We identify $a\in A$ with $(f_a,1)\in G.$ We also have the natural inclusion of $B$ in $G$ by $b \hookrightarrow (1,b)$.

We denote by  $f^b=b^{-1}fb$ the conjugation by elements of $B$ . Then

\begin{equation}
f_a^{b}=b^{-1}f_ab=(1,b^{-1})(f_a,1)(1,b)=(b^{-1}(f_a),b^{-1})(1,b)=(b^{-1}(f_a),1).
\end{equation}

Thus, $D \lhd G$. We also notice that  $b^{-1}(f_a)$ is the function that maps $b^{-1}$ to $a$ and everything else to $1$.

Now let $A$ be a finitely generated  group and $B$ be a finitely generated abelian group.
Let $S_A=\{ a_1, \dotsc,  a_t \}$ and $S_B=\{ b_1, \dotsc,  b_r  \}$ be the sets of generators of $A$ and $B$ respectively.  Then  $S_G=S_A  \cup  S_B$ generates $G$(see \cite{parry } ).

Let $F=F(S_G)$ and  $g \in F $ . $g$ can be expressed as follows
\begin{equation}
g=b_1^{n_{11}}  \dotsm b_r^{n_{1r}}c_1b_1^{n_{21}}  \dotsm b_r^{n_{2r}} c_2 \dotsm b_1^{n_{k1}}  \dotsm b_r^{n_{kr}} c_k b_1^{n'_{1}}  \dotsm b_r^{n'_{r}},
\end{equation}
where $c_1,\dotsc , c_k\in A$.

We note that $b_1^{n_{11}}  \dotsm b_r^{n_{1r}}$ and $b_1^{n'_{1}}  \dotsm b_r^{n'_{r}}$ might be identity.

Let  $h_i=b_1^{n_{i1}}  \dotsm b_r^{n_{ir}}$ for $1 \leq i \leq k$ and $h'=b_1^{n'_{1}} \dotsm b_r^{n'_{r}}$.
 Then, $g=h_1c_1\dotsm h_kc_k h'$. We can rewrite $g$ as follows (using the fact that $B$ is abelian)

\begin{align}
 \nonumber g=& (h_1c_1h_1^{-1})(h_1h_2c_2h_2^{-1}h_1^{-1})(h_1h_2h_3c_3h_3^{-1}h_2^{-1}h_1^{-1})  \\
  \nonumber    &\dotsm  (h_1h_2 \dotsm h_k c_k  h_k^{-1} \dotsm  h_2^{-1}h_1^{-1}) h_1 \dotsm h_k h'     \\
                         & = c_1^{h_1^{-1}} c_2^{h_1^{-1}h_2^{-1}} \dotsm c_k^{h_1^{-1} \dotsm h_k^{-1}} h' h_1 \dotsm h_k  .
\end{align}

Let  $w_i=h_1^{-1} \dotsm  h_i^{-1}$ for $1 \leq i \leq k$ and $b=h' h_1 \dotsm h_k$. Then, $g=c_1^{w_1} \dotsm c_k^{w_k}b$. Since $c_i^{w_i} \in A$ is the function that maps $w_i^{-1}$ to $c_i$ and any other element of $B$ to the  identity, if $w_i \neq w_j$, $c_i^{w_i}c_j^{w_j}=c_j^{w_j}c_i^{w_i}$.  On the other hand, we can assume that $w_i \neq w_j$ for $i\neq j$. To see this suppose that there exist $w_i$ and $w_j$ such that $w_i=w_j$ and $i < j$. Also assume that $j$ is the smallest index satisfying this property. Then, we can flip $c_i^{w_i}$  over $c_{i+1}^{w_{i+1}}, \dotsc ,c_{j-1}^{w_{j-1}}$ :

\begin{align}
\nonumber  g &= c_1^{w_1} \dotsm c_i^{w_i}c_{i+1}^{w_{i+1}} \dotsm c_j^{w_j} \dotsm c_k^{w_k}b\\
\nonumber     &=  c_1^{w_1} \dotsm a_{i+1}^{w_{i+1}} c_i^{w_i}\dotsm c_j^{w_j} \dotsm c_k^{w_k}b\\
\nonumber     & \vdots \\
\nonumber     &=  c_1^{w_1} \dotsm c_{i+1}^{w_{i+1}}\dotsm  c_i^{w_i} c_j^{w_j} \dotsm c_k^{w_k}b \\
                         &=  c_1^{w_1} \dotsm c_{i+1}^{w_{i+1}}\dotsm  (c_ic_j)^{w_j} \dotsm c_k^{w_k}b.
\end{align}

since we are looking for the  minimal number of the generators needed to present $g$ and  $|c_ic_j|_{A} \leq |c_i|_A+ |c_j|_A$, we don't lose anything in terms of the minimality  of $|c_i|_A$. Therefore, we can assume that $w_i \neq w_j$, for all $ i \neq j$.

Now, we consider the Geodesic Problem for $g$ with respect to the set of generators $S_G$.
 Suppose that we know how to solve the Geodesic Problem in $A$ with respect to $S_A$, i.e, we know the Geodesic Length of $c_i$ with respect to $S_A$ denoted by $|c_i|_A$.

For $g=b_1^{n_{11}}  \dotsm b_r^{n_{1r}}c_1b_1^{n_{21}}  \dotsm b_r^{n_{2r}}c_2  \dotsm b_1^{n_{k1}}  \dotsm b_r^{n_{kr}}c_kb_1^{n'_{1}}  \dotsm b_r^{n'_{r}}$ we denote by $l$ the following sum:

\begin{equation}
l=\sum_{i=1}^{r}n_{1i}+ \dotsb  +\sum_{i=1}^{r}n_{ki}+\sum_{i=1}^{r}n'_{i} + \sum_{i=1}^{k}|c_{i}|_A.
\end{equation}

We want to relate this sum to a path in the Cayley graph of $B$. For the moment let assume  that $B$ is free and consider $C=\bZ^r$ viewed as the Cayley graph of $B$ with respect  to the set of  generators $S_B$.  Let  $v_0$ be the vertex of $C$ corresponding  to identity. Let $x=b_1^{x_1}  \dotsm b_r^{x_r}$ and $y=b_1^{y_1}  \dotsm b_r^{y_r}$ be two vertices in $C$. The  Manhattan distance between $x$ and $y$ is defined as

\begin{equation}
d(x,y)=\sum_{i=1}^{r}\arrowvert x_i-y_i \arrowvert .
\end{equation}

In other words, $d(. ,.)$ is the restriction of the distance induced by $\mathcal{L}_1$-norm on $\mathbb {R}^r$ to $\bZ^r$ .

\begin{definition}
 Let $S=\{w_1, \dotsc , w_m , b\}$ be a set of elements of $B$ such that $w_i \neq w_j$ for $i\neq j$.  We say  $l$ is a walk on $S$ with end point $b$ if $l$ is a path of the form  $v_{0}, w_{i_1}, \dotsc,  w_{i_m},b,$ where $w_{i_1}, \dotsc,  w_{i_m}$ is a permutation of $w_1, \dotsc , w_m$.  We call such $l$ minimum if it has the minimum length with respect to $d(.,.)$ among all such walks.
\end{definition}

\begin{lemma}  
\cite{parry }  \label{geo word} 
Let $g=a_{i_1}^{w_1} \dotsm a_{i_k}^{w_k} b \in F$ where $a_{i1},\dotsc , a_{i_k}\in A$,  $w_1, \dotsc w_k \in B$ and $w_i$ are all different. Any {\it geodesic word} $g'$ such that $g'=\mu(g)$ corresponds to a minimum walk in $C$ on $S=\{w_1, \dotsc, w_m , b\}$ with end point $b$.
\end{lemma}

\begin{corollary}\label{reduction}
 If we know how to solve Geodesic Problem in $A$, solving Geodesic Problem in $G=A \wr B$ is equivalent to finding a minimum walk on a finite subgraph of  the Cayley graph of $B$.
\end{corollary}

 {\bf{\em The Traveling Salesman Path Problem in $\bZ^r$($TSPP$) \cite{grid} :}} 
 Given a set of points $\{ s, w_1, \dotsc, w_n, t\} \subset \bZ^r$ with two distinguished points $s,t$  and a metric on $\bZ^r$, find a minimum path of the form $s,w_{i_1}, \dotsc,  w_{i_n}, t$  where $w_{i_1}, \dotsc,  w_{i_n}$ is a permutation of $w_1, \dotsc , w_n$ .

 It has been shown that the Traveling Salesman Problem in $\bZ^r$ with Manhattan distance is $NP$-hard for $n\geq 2$ (see \cite{grid}) . From that, one can easily conclude the $NP$-hardness of $TSPP$ in $\bZ^r$.

\begin{theorem}
 Geodesic Problem in $G=A\wr B$ is $NP$-hard where $A$ is a finitely generated non-trivial group and $B$ is a finitely generated abelian group containing $\bZ^2$.
 \end{theorem}

\begin{proof}
Without loss of generality we can assume that $B$ is free. In order to show $NP$-hardness of the Geodesic Problem in $G$, we will reduce the problem of finding a minimum walk on a finite set of  $C=\bZ^r$, viewed as the Cayley graph of $B$, to the Geodesic Problem in $G$. Let $v_0, w_1, \dotsc, w_k,b$ be $k+2$ different vertices in $C$(we can always assume that $v_0$ is one the points). We want to find a minimum walk on $S=\{w_1, \dotsc, w_k,b\}$ with endpoint $b$.  Now consider the following Geodesic Problem in $G$. Fix $a \in S_A$ and Consider $g=a^{w_1}a^{w_2} \dotsm a^{w_k}b$ . Assume that we know how to solve the Geodesic Problem for $g$. Thus, we can find a { \it geodesic word} $g'$ representing $g$. By lemma \ref{geo word}  g corresponds to a minimum walk $v_0,w_{i_1},w_{i_2}, \dotsc, w_{i_k} b$ on $S$. Thus, if we know how to solve  Geodesic Problem in $G$, we will be able to solve the minimum walk problem in C in polynomial time which is a contradiction. This proves the $NP$-hardness of the Geodesic Problem in $G$ .
\end{proof}

\begin{remark}
The Bounded Geodesic Length Problem($BGLP$)  is the decision version of the Geodesic Problem. An immediate  consequence of the previous theorem is that  if $BGLP$ for $A$ is in class $NP$ then it is  $NP$-complete for $G$.
\end{remark}

\begin{theorem}
If Geodesic Problem in $A$ is polynomial and $B={\mathbb Z}\times {\mathbb Z}_{k_1}\times \dotsb \times {\mathbb Z}_{k_t}$, then Geodesic Problem in $G=A\wr B$ is polynomial.
 \end{theorem}

\begin{proof} 
 A polynomial time algorithm to solve $TSP$ in the Cayley graph of $B={\mathbb Z}\times {\mathbb Z}_{k_1}\times \dotsb \times  {\mathbb Z}_{k_t}$ with the degree of the polynomial equal to $k_1\dotsm k_t$ can be constructed as in \cite{Rothe}.  The statement of the theorem follows now from Lemma 3.1.
There exists even better algorithm for TSP in $B$ with the degree of the polynomial independent of $k_1\dotsm k_t$ \cite{Babai}.
\end{proof}

Theorems 2.3 and 2.4 imply the following result that answers an open question from \cite{myasnikov}.

\begin{corollary} 
Let $G=A \wr B$ where $A,B$ are finitely generated non-trivial abelian groups. Then the Geodesic Problem is $NP$-hard in $G$ if and only if $B$  contains $\bZ^2$.
\end{corollary}

We proved that Geodesic Problem in $G=A \wr B$ is $NP$-hard if $B$ contains ${\mathbb Z}^2$. So trying to get an approximative solution would be natural as the next step. Indeed, for applications a solution that is close to an optimal one
very often is as good as the optimal solution.

{\bf{\em Approximation Algorithm}}.   Let denote by $OPT$ the optimal solution of an minimization problem $P$ . An algorithm $A$ approximates problem $P$ within a factor of $\rho \geq 1$ if  $f (A)/OP T  \leq \rho$, where $f (A)$ is the solution given by $A$ .

{\bf{\em  Polynomial-Time Approximation Scheme  }}.  A $PTAS$  or $Polynomial$-$Time$ $Approximation$ $Scheme$ for a problem is a family of polynomial-time algorithms such that for any given constant $c$, there is an algorithm in the family that approximates the problem within a factor of $(1 + 1/c)$. The running time of the algorithm might depend on $c$, but for each fixed c, it is polynomial in the size of the input.

We state the main result of this section in the next theorem.

\begin{theorem}
Assuming that  the Geodesic Problem in A is solvable in polynomial time, there exists a $PTAS$ for the Geodesic Problem in $G$.
\end{theorem}

In order to prove the theorem  we need the following result.

 \begin{proposition}\cite{arora}
 There exists a $PTAS$ for TSP in $R^d$ with $Euclidean$ norm which generalizes to other $L^p$ norms for $p\geq 1$. A randomized version of the algorithm gives an approximation within a factor of $(1+1/c)$ of the optimal tour  in $\pO(n(logn)^{(\pO (\sqrt(dc) ))^{d-1}}$. If we derandomize the algorithm, we multiply the running time by $\pO(n^d) .$

  The same is true if instead of a TSP tour we consider a TSP  path \cite{arora1}.
\end{proposition}

The theorem is a consequence of corollary \ref{reduction} and the previous proposition. Let $|S_B|=r$ and $g = a_1^{w_1} \dotsm a_k^{w_k} b$. We need to solve the $Traveling$ $Salesman$ $Problem$ walk in the subgraph of $C$ induced by $\{ v_{0},w_1, \dotsc, w_k, b \}$. we consider $C$ with Manhatan distance. The $ PTAS$ finds a word $g'=a_ {i_1} ^{w_{i_1}}a_ {i_2} ^{w_{i_2}} \dotsm a_ {i_k} ^{w_{i_k}}b$ such that $\sum_{j=1}^{k}d(w_{i_j},w_{i_{j-1}})+ \sum_{j=1}^{k} |a_{i_j}| \leq (1+1/c) l_{S_A, S_B}(g)$,where $l_{S_B ,S_A} (g)$ is the Geodesic Length of $g$ with respect to $S_B$ and $S_A$. If we use a randomized version of the algorithm the total running time is

$$\pO(k(logk)^{(\pO (\sqrt(rc)))^{r-1}} + (\text{time needed to find {\it geodesic words} representing } a_i)$$

\section{Geodesic Problem In Free Metabelian groups}

Let $F=F(X)$ be a free group of rank $r$.  Denote by $F'=[F,F]$ the derived subgroup of $F$ and by $F''=[F',F']$ the second derived subgroup of $F$.  Let $G=F/F'$ and $H=F/F''$. $G$ is a free abelian group of rank $r$ and  $H$ is a free metabelian group of rank $r$.

 In \cite{myasnikov} the authors show that $BGLP$ is $NP$-complete for free metabelian groups of arbitrary rank. Since  $BGLP$  is polytime reducible to the $GLP$, $GLP$ is  $NP$-hard in free metabelian groups. In this section we show that there is a $2$-approximation for the Geodesic Problem in $H$. First, we need to remind some definitions.

 {\bf{\em Flow}}.  Let $\gamma=(V,E)$ be a directed graph with two distinguished verices: a source $s$ and a sink $t$.  A flow on $\gamma$ is a function $f : E \rightarrow  \mathbb{R}$ such that

\begin{equation}
\sum_{e:o(e)=v}f(e)-\sum_{e:t(e)=v} f(e)=0  \hspace{1 cm}  \forall v \in V -\{s,t\}.
\end{equation}

The number $f^*(v):= \sum_{e:o(e)=v}f(e)-\sum_{e:t(e)=v} f(e)$ is called the net flow at $v$. The condition above is the same as $f^*(v)=0$ for $v \neq s,t$ and is usually referred to as { \it Kirchhoff} law. We call a flow $f$ circulation If $f^*(v)=0$  also for $s$ and $t$.  In this discussion we only consider $\mathbb{Z}$-flows, e.i., $f: E \rightarrow \bZ$

Let $p$ be a path in $\gamma$  from $v_1$ to $v_2$ and let $\pi_p : E \rightarrow \bZ$ be such that $\pi_p(e)$ is equal to the number of times that $p$ passes through $e$ counted $-1$ when $p$ takes $e$ backward. $\pi_p $ satisfies the flow condition for $v_1$ as source and $v_2$ as sink . Thus, any path in $\gamma$ induces a $\mathbb{Z}$-flow on $\gamma$.

 Let $w \in F$ and let $p_w$ be the path labeled by $w$ in $\gamma$ Cayley graph of $G$ . We denote by $\pi_w$  the flow induced by $p_w$ in $\gamma$.
 
 Let $\mu' : F(X) \rightarrow H$ be the canonical epimorphism. We identify $X$ with its image under $\mu'$($\mu'$ is one to one on $X$). An expression of $l_X(\mu'(w))$ for an arbitrary word $w \in F$ has been given in  \cite{myasnikov} . Before presenting our result we need to remind the definition of  two problems.

{\bf { \em Minimum Steiner Tree Problem:}}  Given a graph $\gamma=(V,E)$ and a subset $V_1 \subset V$, find a minimum subgraph(subtree) such that covers all vertices in $V_1$. 

{\bf { \em Minimum group Stiener Tree Problem:}} Given connected components $C_i$ of $\gamma$, find a minimum subgraph that makes the subgraph $\cup C_i$ connected.

Let  $w \in F$ and $\pi_w$ be the flow induced by $w$ in $\gamma$.  We define $supp(\pi_w)=\{e \in E | \pi(e)\neq 0\}$. A minimum Group  Steiner  Tree for $w$ is a minimal Group Steiner  Tree for the connected components of the subgraph induced by $supp(\pi_w)$ in $\gamma$.

\begin{proposition}
 \cite{myasnikov}  Let $\gamma$ be the cayley graph of $G$ with respect to the set of generators $X$. then for $w \in F$ we have

\begin{equation} \label{geo}
 l_X(\mu'(w))= \sum_{e \in  \text{supp}(\pi_w)} \pi_w(e)+2|E(Q)|,
\end{equation}

where  $\pi_w$  is the path induced by $w$ in $\gamma$ and $Q$ is a minimum group steiner tree for $w$ in $\gamma$.
\end{proposition}

In the expression for $l_X(\mu'(w))$ we can evaluate $\sum_{e \in  \text{supp}(p_w)} \pi_w(e)$ in polynomial time. The time-consuming part is finding $Q$.  Unfortunately, there has been no $PTAS$ known so far for the $Group$ $Steiner$ $Tree $ problem. Nevertheless, in the same paper a similar statement to the $PTAS$ result for $Euclidean$ $Traveling$ $Salesman$ $Problem$ is proved for $Minimum$ $Steiner$ $Tree$ $Problem$.

 \begin{proposition}
 \cite{arora} There exists a $PTAS$ for the Euclidean Minimum Steiner Tree in $R^d$ which generalizes to other $L^p$ norms for $p\geq 1$. A randomized version of the algorithm gives an approximation within a factor of $(1+1/  c)$ of the optimal tree in  $\pO(n(logn)^{(\pO (\sqrt(dc) ))^{d-1}}$. If we derandomize the algorithm, we multiply the running time by $\pO(n^d) $.
\end{proposition}

\begin{theorem} 
There is a $2$-approximation algorithm for the Geodesic Problem in the free metabelian group of arbitrary rank $r$.
\end{theorem}

\begin{proof} 
Let $C_i$ be the connected components of the subgraph induced by $supp(\pi_w)$ in $\gamma$.  We choose an arbitrary point  $y_i$ in each  $C_i$.  Let $Q_1$ be the $(1+1/c)$-approximation of the Minimum Steiner Tree Problem for $\{ y_i \}$. So $ |Q_1|/(1+1/c)$ is the optimal value for this problem.   We denote by $Q^*$ the optimal solution for the Minimum Group Steiner Tree on $C_i$.  $Q^*  \bigcup (\cup C_i) $ is a connected subgraph of $\gamma$ containing $\{ y_i \}$(see Fig. \ref{fig} ). Thus

\begin{equation}
\frac{|Q_1|}{(1+1/c)}  \leq |Q^*  \bigcup (\cup C_i)| = |Q^*| + \sum |C_i|.
\end{equation}

So

\begin{equation}
|Q_1|-|Q^*| \leq 1/c |Q^*|+(1+1/c)\sum |C_i|.
\end{equation}

\begin{figure}[ph] 
\centerline{\psfig{file=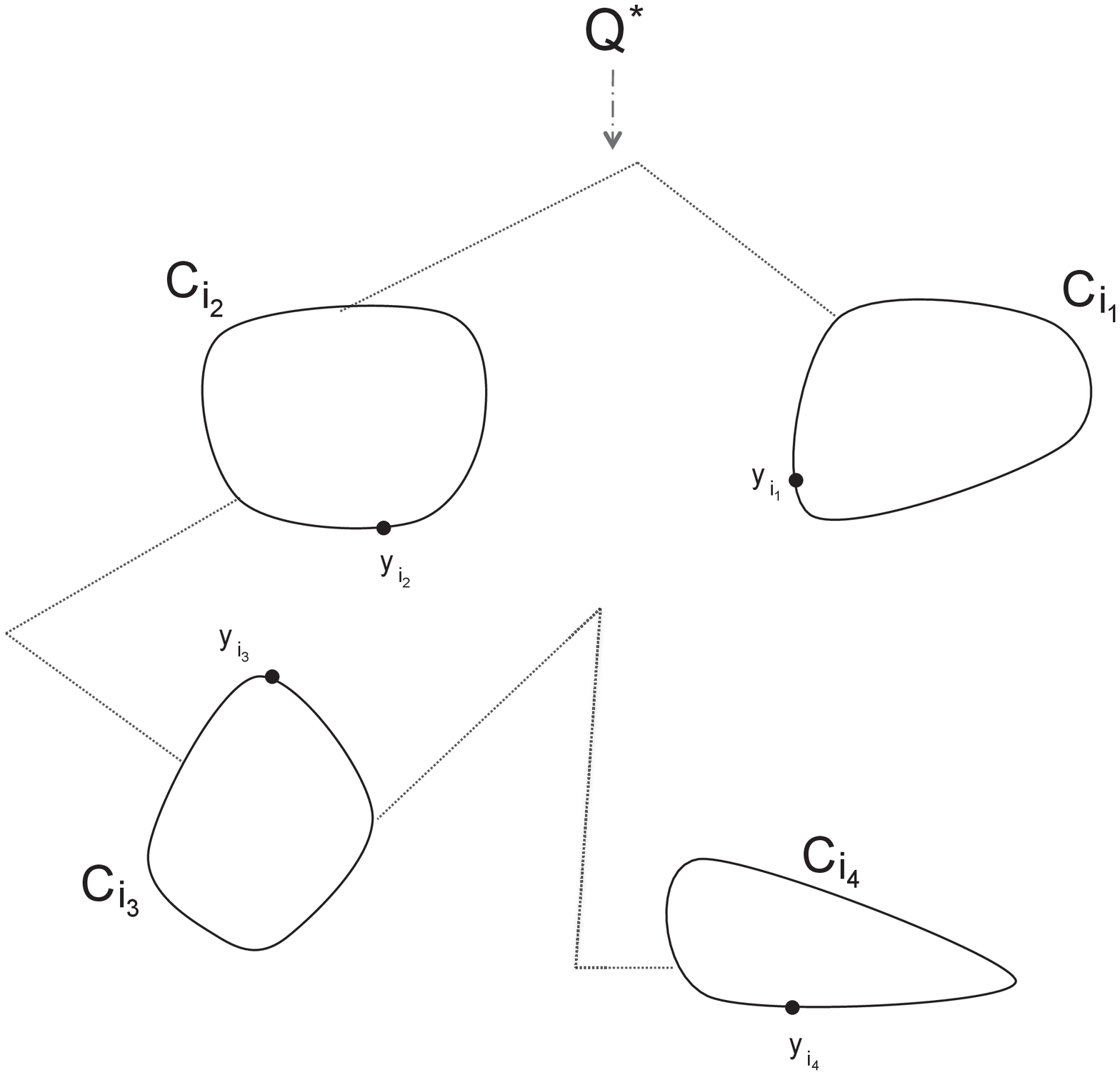,width=4.5in}} 
\vspace*{8pt}
\caption{\label{fig}}
\end{figure}

If we use $Q_1$ instead of $Q^*$ in formula \ref{geo}, we get an approximation $l'(w)$ for $l(w)$ such that

\begin{align}
   \nonumber  \frac{l'(w)-l(w)}{l(w)} & = \frac{(\sum_{e \in \text{ supp}(\pi_w) } \pi_w(e)+2|Q_1|)-(\sum_{e \in \text{ supp}(\pi_w) } \pi_w(e)+2|Q^*|)}{\sum_{e \in \text{ supp}(\pi_w) } \pi_w(e)+2|Q^*|)} \\
    \nonumber                                     &=\frac{2(|Q_1|-|Q^*|)}{\sum_{i } |C_i|+2|Q^*|)}\\
                                                             & \leq  \frac{2(1/c |Q^*|+(1+1/c)\sum |C_i|)}{\sum_{i } |C_i|+2|Q^*|)}.
\end{align}

If we choose $c$ such that $1/c < \frac{1}{\sum |C_i|)}$, then

\begin{align}
    \nonumber  1/c |Q^*|+ 1/c \sum |C_i| & < 1/c |Q^*|+1  \\
                                                                     &\leq 2 |Q^*|.
\end{align}

 The last inequality is satisfied since $ |Q^*| > 0$, otherwise there is nothing to prove. This implies

\begin{equation}
 \frac{1/c |Q^*|+(1+1/c)\sum |C_i|}{\sum_{i } |C_i|+2|Q^*|)} <1.
 \end{equation}
 
 So
 
\begin{align}
   \nonumber  \frac{l'(w)-l(w)}{l(w)} & \leq  \frac{2(1/c |Q^*|+(1+1/c)\sum |C_i|)}{\sum_{i } |C_i|+2|Q^*|)}\\
                                                             &  < 2.
\end{align}

which proves that $l'(w)$ is a $2$-approximation of $l(w)$.

\end{proof}

\end{document}